\documentclass[12pt,a4paper]{amsart}
\usepackage[margin=1.5cm,includefoot,footskip=30pt]{geometry}
\usepackage{amsaddr}
\usepackage{amssymb}
\usepackage[all,cmtip]{xy}
\usepackage{cite}
\usepackage[utf8]{inputenc}
\usepackage{upgreek}
\usepackage{fix-cm}
\usepackage{graphicx}
\usepackage[T1]{fontenc}
\long\def\symbolfootnote[#1]#2{\begingroup%
\def\thefootnote{\fnsymbol{footnote}}\footnote[#1]{#2}\endgroup}

\newcommand{\tr}{\ensuremath{{}^T\!\!}}
\newcommand{\tra}{\ensuremath{{}^T}\!}
\newcommand{\C}{\mathfrak C}

\newcommand{\E}{\mathcal E}
\newcommand{\F}{\mathcal F}

\newcommand{\diag}{\textup{diag}}

\newtheorem{theorem}{Theorem}[section]
\newtheorem*{theorema}{Theorem A}
\newtheorem*{theoremb}{Theorem B}
\newtheorem{lemma}[theorem]{Lemma}
\newtheorem{corollary}[theorem]{Corollary}

\newtheorem{definition}{Definition}[section]
\newtheorem*{theorem*}{Theorem}

\numberwithin{equation}{section}

\newcommand{\ignore}[1]{}

\newcommand{\mynote}[1]{}

\makeatletter
\def\Ddots{\mathinner{\mkern1mu\raise\p@
\vbox{\kern7\p@\hbox{.}}\mkern2mu
\raise4\p@\hbox{.}\mkern2mu\raise7\p@\hbox{.}\mkern1mu}}
\makeatother
\begin{document}
\title[Gaussian elimination in  unitary groups]{Gaussian elimination in unitary groups with an application to cryptography}
\author[Mahalanobis and Singh]{Ayan Mahalanobis and Anupam Singh} 
\address{IISER Pune, Dr.~Homi Bhabha Road, Pashan, Pune 411008, INDIA.} 
\email{ayan.mahalanobis@gmail.com}
\email{anupamk18@gmail.com}
\thanks{This work was supported by a SERB research grant MS: 831/13.}


\begin{abstract} Gaussian elimination is used in special linear groups to
  solve the word problem. In this paper, we extend Gaussian
  elimination to unitary groups. These algorithms
  have an application in building a public-key cryptosystem, we demonstrate that.
\end{abstract}
\today
\maketitle
\section{Introduction}
Gaussian elimination is a very old theme in computational
mathematics. It was developed to solve linear simultaneous
equations. The modern day matrix theoretic approach was developed by John
von Neumann and the popular textbook version by Alan Turing. Gaussian elimination has many applications
and is a very well known mathematical method. We will not elaborate on it
any further, but will refer an interested reader to a nice article by Grcar~\cite{grc}. The way we look at Gaussian elimination is: it gives us an
algorithm to write any matrix of the \emph{general linear group},
GL$(d,\mathcal{K})$, of size $d$ over a field $\mathcal{K}$ as the product of elementary
matrices and a diagonal matrix with all ones except one entry, using
elementary operations. That
entry in the diagonal is the determinant of the matrix. There are many ways to look at
this phenomena. One simple way is: one can write the \emph{matrix
as a word in generators}. So in the language of computational group
theory the word problem in GL$(d,\mathcal{K})$ has an efficient algorithm --
Gaussian elimination.

We write this paper to say that one can have a very similar result
with unitary groups as well. We define
\textbf{elementary matrices} and \textbf{elementary operations} for unitary groups. These
matrices and operations are similar to that of \emph{elementary transvections}
and \emph{elementary row-column operations} for special linear
groups. Using these elementary matrices and elementary operations, we solve the word problem in unitary groups in a way that is very similar to the general linear groups. Similar algorithms are being developed for other classical groups and will be presented elsewhere.

In this paper, we work with a different set of generators than that is
usual in computational group theory. The usual generators are called
the \emph{standard generators}~\cite[Tables 1\,\&\,2]{lo}. Our
generators, we call them elementary matrices and are defined later,
have their root in the root spaces in Lie theory~\cite[Sections 11.3, 14.5]{ca} and have the disadvantage of being a larger set compared to that of the standard generators. However, standard generators being \lq\lq{multiplicative\rq\rq{} in nature, depends on the primitive element of a finite field, works only for finite fields. On the other hand, our generators, work for arbitrary fields. Using standard generators, one needs to solve the discrete logarithm problem often. No such need arises in our case. In the current literature, the best row-column operations in unitary groups is by Costi~\cite{costi} and implemented in Magma~\cite{magma} by Costi and C.~Schneider. Using their magma function \emph{ClassicalRewriteNatural}, we show that our algorithm is much faster, see Figure 1.

A need for row-column operations in classical groups was articulated by Seress~\cite[Page
 677]{seress} in 1997. Computational group theory and in particular 
\emph{constructive recognition of classical groups} have come a long way till then.
 We will not give a historical overview of this, an interested reader 
can find such an overview in the works of Brooksbank~\cite[Section 1.1]{brooksbank1},
Leedham-Green and O'Brien~\cite[Section 1.3]{lo} and O'Brien~\cite{ob1}. Two recent works that are relevant to our work are Costi~\cite{costi} and Ambrose et.~al.~\cite{csaba}.  
 
In this paper, we only deal with unitary groups defined by the Hermitian form $\beta$ defined later.
The Hermitian form for the even-order case works for \textbf{all characteristic}. However, in the odd-order case the $2$ in the upper-left makes it useless in the even characteristic. One can change this $2$ to a $1$ in $\beta$, however, then one needs to compensate that by putting $\frac{1}{2}$ in the generators. We tried, but were unable to extend our algorithm for the odd-order unitary group to even characteristic. For even-order unitary groups, the algorithm developed in this paper works for all characteristic. However, for the odd-order case only \textbf{odd characteristic will be considered}. 

\subsection{Notations} For the rest of the paper, let $\mathcal{K}$ be a quadratic extension of the field $k$.There is an automorphism of degree two involved with these extensions and will be denoted by $\sigma:x\mapsto\bar{x}$. In the case of $\mathbb{C}:\mathbb{R}$, $\sigma$ is the complex conjugation. In the case of a finite field $\mathbb{F}_{q^2}:\mathbb{F}_q$, $\sigma$ is the map $x\mapsto x^q$. We fix a non-zero $\varepsilon\in\mathcal{K}$ with $\bar\varepsilon=-\varepsilon$. Then every $x\in\mathcal{K}$ is of
the form $x=a+\varepsilon b$. We denote $\mathcal{K}^o=\{x\in\mathcal{K}\mid \bar x=-x\}$. We also denote $\mathcal{K}^1=\{x\in\mathcal{K}\mid x\bar x=1\}$. A $d\times d$ matrix $X$ is called Hermitian (skew-Hermitian) if $\tra{\bar X}=X$ ($\tra{\bar X}=-X$). Two important examples of $\mathcal{K}:k$ pairs that we have in mind are $\mathbb{C}:\mathbb{R}$ and $\mathbb{F}_{q^2}:\mathbb{F}_q$. 

The main result that we prove in this paper follows. The result is well known, however the algorithmic proof of the result is original. Moreover, this algorithm is of independent interest in other areas, for example, constructive recognition of classical groups.
\begin{theorema}\label{theorema}
For $d\geq 4$, using elementary operations, one can write any matrix $\mathcal{A}$ in $\text{U}(d,\mathcal{K})$,
the unitary group of size $d$ over $\mathcal{K}$, as product of
elementary matrices and a diagonal matrix. The diagonal matrix is of
the following form:
\begin{itemize}
\item $\left(
\begin{array}{cccc|cccc}
1&&&&&&&\\
&\ddots&&&&&&\\
&&1&&&&&\\
&&&\lambda&&&&\\
\hline
&&&&1&&&\\
&&&&&\ddots&&\\
&&&&&&1&\\
&&&&&&&\bar{\lambda}^{-1}
\end{array}\right)$where
  $\lambda\bar{\lambda}^{-1}=\det{\mathcal{A}}$ and $d=2l$.

\item $\left(
\begin{array}{c|cccc|cccc}
\alpha&&&&&&&&\\
\hline
&1&&&&&&&\\
&&\ddots&&&&&&\\
&&&1&&&&&\\
&&&&\lambda&&&&\\
\hline
&&&&&1&&&\\
&&&&&&\ddots&&\\
&&&&&&&1&\\
&&&&&&&&\bar{\lambda}^{-1}
\end{array}\right)$where $\alpha\bar\alpha=1$ and
  $\alpha\lambda\bar{\lambda}^{-1}=\det{\mathcal{A}}$ and $d=2l+1$.
\end{itemize}
Here $\bar{\lambda}$ is the image of $\lambda$ under the automorphism $\sigma$.
\end{theorema}
 
 A trivial corollary (Theorem~\ref{main-cor}) of our algorithm is very similar to
 a result by Steinberg~\cite[\S6.2]{st3}, where he describes the generators
 of a projective-unitary group over odd characteristic. Our work is somewhat
 similar in nature to the work of Cohen~et.~al.~\cite{CMT}, where the
 authors study generalized row-column operations in Chevalley
 groups. They did not study twisted groups. 

We use the algorithm developed to construct a MOR
cryptosystem in unitary groups and study its security. From the
discussion in Section 6 of this paper it follows:
\begin{theoremb}
The security of the MOR cryptosystem over U$(d,q^2)$ is equivalent to
the hardness of the discrete logarithm problem in $\mathbb{F}_{q^{2d}}$.
\end{theoremb} 
\section{Unitary Groups}\label{chevalleygroups}

One of the legendary works of Chevalley~\cite{ch} is a way to construct groups over an arbitrary field from a complex simple Lie algebra. These groups are now called Chevalley groups in his honor. Steinberg~\cite{st3} generalized Chevalley\rq{}s idea and introduced twisted Chevalley groups. These groups are now called \emph{Steinberg groups}. These groups can be constructed in those cases where the Dynkin diagram of the underlying simple Lie algebra has a non-trivial symmetry. In this paper we work with the twisted group of type $^{2}A_{l}$, i.e., unitary groups.

Let $\mathcal{K}$ be a field with a non-trivial field automorphism $\sigma$ of order $2$ with fixed field $k$. 
Let $V$ be a vector space of dimension $d$ over $\mathcal{K}$. We denote the image of $\alpha$ under $\sigma$ by $\bar \alpha$. Let $\beta \colon V\times V \rightarrow \mathcal{K}$ be a non-degenerate Hermitian form,
i.e., bar-linear in the first coordinate and linear in the second coordinate satisfying $\beta(x,y)=\overline{\beta(y,x)}$. We fix a
basis for $V$ and slightly abuse the notation to denote the matrix of $\beta$ by $\beta$. Thus $\beta$ is a non-singular matrix satisfying $\beta=\tra{\bar\beta}$. 
\begin{definition}[Unitary Group]
The unitary group is:
$$\text{U}(d,\mathcal{K})=\{X\in \text{GL}(d,\mathcal{K}) \mid \tr{\bar X}\beta X=\beta\}.$$
The special unitary group $\text{SU}(d,\mathcal{K})$ consists of matrices of $\text{U}(d,\mathcal{K})$ of determinant $1$.
\end{definition}
In this paper we work with split (i.e., maximum Witt index) Hermitian form. Recall that characteristic of $\mathcal{K}$ is odd whenever $d$ is odd.  
 For the convenience of computations we index the basis by
$1,\ldots, l, -1,\ldots,-l $ when $d=2l$ and by $0,1,\ldots,l,
-1,\ldots, -l$ when $d=2l+1$; where $l>1$. We also fix the matrix $\beta$ as follows:
\begin{itemize}
\item $d=2l$ fix $\beta=\begin{pmatrix} & I_l \\ I_l & \end{pmatrix}$
\item $d=2l+1$ fix $\beta = \begin{pmatrix} 2 & & \\ &&I_l \\ &I_l&\end{pmatrix}$.
\end{itemize}
Thus the unitary group obtained with respect to this form is called the split unitary group.

There are two important examples of this split unitary group and subsequently of our algorithm: the field of complex numbers $\mathbb C$ over reals $\mathbb R$ with $\sigma$ the complex conjugation and the other, finite field $\mathbb F_{q^2}$ over $\mathbb F_q$ with sigma being $\alpha\mapsto \alpha^q$. It is known that in both the cases there is only one
non-degenerate Hermitian form up to equivalence~\cite[Corollary
10.4]{gr}. Equivalent Hermitian forms gives rise to conjugate unitary groups.
In the case of a finite field, a unitary group will be denoted by $\text{U}(d,q^2)$ and
special unitary group as $\text{SU}(d,q^2)$. A word of caution: in the literature
$\text{U}(d,q^2)$, $\text{U}(d,\mathbb F_q)$ and
$\text{U}(d,q)$ are used interchangeably. 

\section{Elementary matrices and elementary operations in unitary groups}\label{elementarymatrices}
Solving the word problem in any group is of interest in computational
group theory. In a special linear group, it can be easily solved
using Gaussian elimination. However, for many groups, it is a very
hard problem. In this paper we present a fast, cubic-time solution to
the word problem in unitary groups. 

Gaussian elimination in SL$(d,\mathcal{K})$ uses elementary transvections as the
elementary matrices and row-column operations as elementary
operations. These elementary operations are multiplication by
elementary matrices. The elementary matrices are of the form $I+te_{i,j} \;\;(t\in \mathcal{K})$, where
$e_{i,j}$ is the matrix unit with $1$ in the $(i,j)\textsuperscript{th}$
position and zero elsewhere. 

In the same spirit, one can define
Chevalley-Steinberg generators for the unitary group~\cite[Section 14.5]{ca} as follows:  
\subsection{Elementary matrices for $\text{U}(2l,\mathcal{K})$}\label{generators} 
In what follows, $l\geq 2$.
For $1\leq i,j\leq l$, $t\in\mathcal{K}$ and $s\in \mathcal{K}^o$:
\begin{eqnarray*}
x_{i,j}(t)=&I+te_{i,j}-\bar t e_{-j,-i} &\text{for}\; i\neq j,\\
x_{i,-j}(t)=&I+te_{i,-j}-\bar t e_{j,-i} &\text{for}\;i<j,\\
x_{-i,j}(t)=&I+te_{-i,j}- \bar t e_{-j,i}&\text{for}\; i<j, \\
x_{i,-i}(s)=&I+se_{i,-i},\\
x_{-i,i}(s)=&I+se_{-i,i},\\
\end{eqnarray*}
\subsection{Row-Column operations for $\text{U}(2l,\mathcal{K})$}
Rephrasing the earlier definition in matrix format, we have three kind of elementary matrices. 
\begin{enumerate}
\item[E1:] $\begin{pmatrix}R&\\ & \tra {\bar R}^{-1}\end{pmatrix}$ where $R=I+te_{i,j}$; $i\neq j $.
\item[E2:] $\begin{pmatrix} I& R \\ &I\end{pmatrix}$ where $R$ is either $te_{i,j}-\bar t e_{j,i}$; $i<j$ or $se_{i,i}$.
\item[E3:] $\begin{pmatrix} I&  \\ R &I\end{pmatrix}$ where $R$ is either $te_{i,j}-\bar t e_{j,i}$; $i<j$ or $se_{i,i}$.
\end{enumerate}
Let $g=\begin{pmatrix}A&B\\C&D\end{pmatrix}$ be a $2l\times 2l$ matrix written in block form of size $l\times l$. Note the effect of multiplying $g$ by matrices from above.

\begin{eqnarray*}
ER1:&\begin{pmatrix}R&\\ & \tra {\bar R}^{-1}\end{pmatrix}\begin{pmatrix}A&B\\C&D\end{pmatrix}&=\begin{pmatrix}RA&RB \\ \tra {\bar R}^{-1}C & \tra {\bar R}^{-1} D\end{pmatrix}.\\
EC1:&\begin{pmatrix}A&B\\C&D\end{pmatrix}\begin{pmatrix}R&\\ & \tra {\bar R}^{-1}\end{pmatrix}&=\begin{pmatrix} AR&B\tra {\bar R}^{-1} \\ CR & D\tra {\bar R}^{-1}\end{pmatrix}.
\end{eqnarray*}
\begin{eqnarray*}
ER2: & \begin{pmatrix}I&R\\ &I\end{pmatrix}\begin{pmatrix}A&B\\C&D\end{pmatrix}&=\begin{pmatrix}A+RC&B+RD \\ C & D\end{pmatrix}. \\
EC2: &\begin{pmatrix}A&B\\C&D\end{pmatrix}\begin{pmatrix}I&R\\ &I \end{pmatrix}&=\begin{pmatrix} A&AR+B \\ C & CR+ D\end{pmatrix}.
\end{eqnarray*}
\begin{eqnarray*}
ER3: &\begin{pmatrix}I&\\R &I\end{pmatrix}\begin{pmatrix}A&B\\C&D\end{pmatrix}&=\begin{pmatrix}A&B \\ RA+ C & RB+D\end{pmatrix}.\\ 
EC3: &\begin{pmatrix}A&B\\C&D\end{pmatrix}\begin{pmatrix}I&\\ R&I \end{pmatrix}&=\begin{pmatrix} A+BR&B \\ C+DR & D\end{pmatrix}.
\end{eqnarray*}
\subsection{Elementary matrices for $\text{U}(2l+1,\mathcal{K})$}
For $l\geq 2$, $1\leq i,j\leq l$, $t\in \mathcal{K}$, $s\in\mathcal{K}^o$ and characteristic of $\mathcal{K}$ odd
\begin{eqnarray*}
x_{i,j}(t)=&I+te_{i,j}-\bar t e_{-j,-i} &\text{for}\; i\neq j,\\
x_{i,-j}(t)=&I+te_{i,-j}-\bar t e_{j,-i} &\text{for}\;i<j,\\
x_{-i,j}(t)=&I+te_{-i,j}- \bar t e_{-j,i}&\text{for}\; i<j, \\
x_{i,-i}(s)=&I+se_{i,-i},\\
x_{-i,i}(s)=&I+se_{-i,i},\\
x_{i,0}(t)=&I-2\bar te_{i,0}+ te_{0,-i}-t\bar te_{i,-i},\\
x_{0,i}(t)=&I+te_{0,i}-2\bar te_{-i,0}-t\bar te_{-i,i},\\
\end{eqnarray*}

\subsection{Row-Column operations for $\text{U}(2l+1,\mathcal{K})$}
Rephrasing in matrix format:
\begin{enumerate}
\item[E1:] $\begin{pmatrix}1&&\\ &R&\\ && \tra \bar R^{-1}\end{pmatrix}$ where $R=I+te_{i,j}$; $i\neq j$.
\item[E2:] $\begin{pmatrix} 1&&\\ &I& R \\ &&I\end{pmatrix}$ where  $R$ is either $te_{i,j}-\bar te_{j,i}$; $i < j$ or $se_{i,i}$.
\item[E3:] $\begin{pmatrix} 1&&\\ &I&  \\ &R &I\end{pmatrix}$ where
  $R$ is either $te_{i,j}-\bar te_{j,i}$; $i < j$ or $se_{i,i}$.
\item[E4:] $\left\{\begin{array}{c}
\begin{pmatrix} 1&&R\\-2\bar{R} &I&-\tr R\bar{R}  \\ &   &I\end{pmatrix}\ \text{where} R=te_i \\
\begin{pmatrix} 1&R&\\&I&  \\-2\bar{R} &-\tr R\bar{R} &I\end{pmatrix}\ \text{where}\ R=te_i
\end{array}\right.$
\end{enumerate}
Here $e_i$ is the row vector with $1$ at $i^{\text{th}}$ place and zero elsewhere.
Let $g=\begin{pmatrix}\alpha&X&Y\\ E& A&B\\F&C&D\end{pmatrix}$ be a $(2l+1)\times (2l+1)$ matrix where $A,B,C,D$ are $l\times l$ matrices. The matrices $X=(X_1,X_2,\ldots,X_l)$, $Y=(Y_1,Y_2,\ldots,Y_l)$, $E=\tra(E_1,E_2,\ldots,E_l)$ and $F=\tra(F_1,F_2,\ldots,F_l)$ are rows of length $l$. Furthermore $\alpha\in\mathcal{K}$. Note the effect of multiplication by elementary matrices from above is as follows:
\begin{eqnarray*}
ER1:&\begin{pmatrix}1&&\\ &R&\\ && \tra {\bar R}^{-1}\end{pmatrix}\begin{pmatrix}\alpha & X&Y\\E& A&B\\F& C&D\end{pmatrix}&=\begin{pmatrix}\alpha &X&Y\\ RE&RA&RB \\ \tra {\bar R}^{-1}F&\tra {\bar R}^{-1}C & \tra {\bar R}^{-1} D\end{pmatrix}.\\
EC1:&\begin{pmatrix}\alpha & X&Y\\ E&A&B\\F&C&D\end{pmatrix}\begin{pmatrix}1&&\\ &R&\\ && \tra {\bar R}^{-1}\end{pmatrix}&=\begin{pmatrix} \alpha& XR & Y\tra {\bar R}^{-1}\\ E&AR&B\tra{\bar R}^{-1} \\ F& CR & D\tra{\bar R}^{-1}\end{pmatrix}.
\end{eqnarray*}
\begin{eqnarray*}
ER2: & \begin{pmatrix}1&&\\ &I&R\\ &&I\end{pmatrix}\begin{pmatrix}\alpha &X&Y\\ E&A&B\\F&C&D\end{pmatrix}&=\begin{pmatrix}\alpha &X&Y\\ E+RF& A+RC&B+RD \\ F& C & D\end{pmatrix}. \\
EC2:&\begin{pmatrix}\alpha& X&Y\\ E& A&B\\F& C&D\end{pmatrix}\begin{pmatrix}1&&\\ &I&R\\ &&I \end{pmatrix}&=\begin{pmatrix} \alpha& X& XR+Y\\ E&A&AR+B \\ F& C & CR+ D\end{pmatrix}.
\end{eqnarray*}
\begin{eqnarray*}
ER3: &\begin{pmatrix}1&&\\ &I&\\&R &I\end{pmatrix}\begin{pmatrix}\alpha & X&Y\\ E&A&B\\F&C&D\end{pmatrix}&=\begin{pmatrix}\alpha&X&Y\\ E&A&B \\ RE+F&RA+ C & RB+D\end{pmatrix}.\\ 
EC3:&\begin{pmatrix}\alpha& X&Y\\ E&A&B\\F&C&D\end{pmatrix}\begin{pmatrix}1&&\\ &I&\\ &R&I \end{pmatrix}&=\begin{pmatrix}\alpha& X+YR&Y\\E& A+BR&B \\ F&C+DR & D\end{pmatrix}.
\end{eqnarray*}
For E4 we only write the equations that we need later.
\begin{itemize}
\item Let the matrix $g$ has $C=\diag(d_1,\ldots,d_l)$.
$$ER4:\;[(I+te_{0,-i}-2\bar te_{i,0}-t\bar te_{i,-i})g]_{0,i}= X_i+td_i $$
$$EC4:\;[g(I+te_{0,-i}-2\bar te_{i,0}-t\bar te_{i,-i})]_{-i,0}= F_i-2\bar td_i.$$
\item Let the matrix $g$ has $A=\diag(d_1,\ldots,d_l)$.
$$ER4:\;[(I+te_{0,i}-2\bar te_{-i,0}-t\bar te_{-i,i})g]_{0,i}= X_i+td_i $$
$$EC4:\;[g(I+te_{0,-i}-2\bar te_{i,0}-t\bar te_{i,-i})]_{i,0}= E_i-2\bar td_i.$$
\end{itemize}

\subsection{Row-interchange matrices}
We need certain row interchange matrices, multiplication with these matrices from left,  interchanges $i^{th}$ row with $-i^{th}$ row for $1\leq i\leq l$. These are certain Weyl group elements. These matrices can be produced as follows: for $s\in\mathcal{K}^o$,
$$w_{i,-i}(s)=x_{i,-i}(s)x_{-i,i}(-s^{-1})x_{i,-i}(s)=I+se_{i,-i}-e_{i,i}-s^{-1}e_{-i,i}-e_{-i,-i}.$$
Note that our row interchange multiplies one row by $s$ and the other by $-s^{-1}$ and then swaps them. This scalar multiplication of rows produce no problem for our cause.
\section{Gaussian elimination in Unitary Group}\label{wordproblem}
Now we present the main result of this paper, two algorithms, one for even-order unitary groups and other for the odd-order unitary groups.

\subsection{The algorithm for even-order unitary groups}
\begin{itemize}
\item[Step 1:]
\textbf{Input}: Matrix $g=\begin{pmatrix}A&B\\C&D\end{pmatrix}$ which belongs to $\text{U}(2l,\mathcal{K})$. 

\noindent\textbf{Output}: Matrix $g_1=\begin{pmatrix}A_1&B_1\\C_1&D_1\end{pmatrix}$ which is one of the following kind: 
\begin{enumerate} 
\item[a:] The matrix $A_1$ is $\diag(1,1,\ldots, 1, \lambda)$ with $\lambda\neq 0$ and $C_1$ is $\begin{pmatrix} C_{11}& -\lambda \tra {\bar C}_{21}\\ C_{21}&c\end{pmatrix}$ where $C_{11}$ is a skew-Hermitian matrix of size $l-1$ and $\bar \lambda c\in\mathcal{K}^o$. 
\item[b:] The matrix $A_1$ is $\diag(1,1,\ldots,1,0,\ldots,0)$ with
  number of $1$s equal to $m<l$ and $C_1$ is of the form
  $\begin{pmatrix}C_{11}&0\\ C_{21}& C_{22}\end{pmatrix}$ where
  $C_{11}$ is an $m\times m$ skew-Hermitian matrix.
\end{enumerate}

\noindent{\textbf{Justification}}: Observe that the effect of ER1 and EC1 on $A$ is the
usual row-column operations on a $l\times l$ matrix. Thus we can reduce $A$ to the diagonal form using classical Gaussian elimination algorithm 
and Corollary~\ref{corA} makes sure that $C$ has the required form.

\item[Step 2:]
\textbf{Input:} Matrix $g_1=\begin{pmatrix}A_1&B_1\\C_1&D_1\end{pmatrix}$.

\noindent\textbf{Output:} Matrix $g_2=\begin{pmatrix}A_2&B_2\\0&\tra {\bar A}_2^{-1}\end{pmatrix}$; $A_2$ is $\diag(1,1,\ldots,1,\lambda)$.

\noindent\textbf{Justification}: Observe the effect of ER3. It changes $C_1$ by $C_1+RA_1$. 
Using Lemma~\ref{lemmaC} we can make the matrix $C_1$ the zero matrix in the first case and $C_{11}$ the zero matrix in the second case. After that we make use of row-interchange matrices to interchange the rows so that we get zero matrix at the place of $C_1$. If required use ER1 and EC1 to make $A_1$ a diagonal matrix, say $A_2$. Lemma~\ref{lemmaB} ensures that $D_1$ becomes $\tra\bar A^{-1}_2$. 

\item[Step 3:] 
\textbf{Input:}  Matrix $g_2=\begin{pmatrix}A_2&B_2\\0& \bar
  A_2^{-1}\end{pmatrix}$; $A_2$ is $\diag(1,\ldots,1,\lambda)$.

\noindent\textbf{Output:} Matrix $g_3=\begin{pmatrix}A_2&0\\0&\bar A_2^{-1}\end{pmatrix}$; $A_2$ is $\diag(1,1,\ldots,1,\lambda)$.

\noindent\textbf{Justification}: Using Corollary~\ref{corB} we see that the
matrix $B_2$ has certain form. We can use ER2 to make the matrix
$B_2$ a zero matrix because of Lemma~\ref{lemmaC}.
\end{itemize}

\subsection{The Algorithm for odd-order unitary groups} Recall that for odd-order unitary groups characteristic of $\mathcal{K}$ is
 odd. The algorithm is as follows: 
\begin{itemize}
\item[Step 1:] 
\textbf{Input:} Matrix $g=\begin{pmatrix}\alpha&X&Y\\ E&A&B\\F&C&D\end{pmatrix}$ which belongs to $\text{U}(2l+1,\mathcal{K})$; 

\noindent\textbf{Output:} Matrix $g_1=\begin{pmatrix}\alpha & X_1 &Y_1 \\ E_1& A_1 &B_1\\ F_1& C_1&D_1\end{pmatrix}$ of one of the following kind:
\begin{enumerate} 
\item[a:] Matrix $A_1$ is a diagonal matrix $\diag(1,\ldots, 1, \lambda)$ with $\lambda\neq 0$. 
\item[b:] Matrix $A_1$ is a diagonal matrix $\diag(1,\ldots,1,0,\ldots,0)$ with number of $1$s equal to $m$ and $m<l$.
\end{enumerate} 

\noindent\textbf{Justification}: Using ER1 and EC1 we can do row and column operations on $A$ and get the required form.
 
\item[Step 2:] 
\textbf{Input:} Matrix $g_1=\begin{pmatrix}\alpha & X_1 & Y_1 \\ E_1& A_1 &B_1\\ F_1& C_1&D_1 \end{pmatrix}$. 

\noindent\textbf{Output:} Matrix $g_2=\begin{pmatrix} \alpha_2 & 0 & Y_2 \\ E_2& A_2 &B_2\\ F_2& C_2& D_2\end{pmatrix}$ of one of the following kind:    
\begin{enumerate} 
\item[a:] Matrix $A_2$ is $\diag(1,1,\ldots, 1, \lambda)$ with $\lambda\neq 0$, $E_2=0$ and $C_2$ is of the form $\begin{pmatrix} C_{11}& -\lambda \tra{\bar C}_{21}\\ C_{21}&c\end{pmatrix}$ where $C_{11}$ is skew-Hermitian of size $l-1$ and $\bar \lambda c\in \mathcal{K}^o$. 
\item[b:] Matrix $A_2$ is $\diag(1,\ldots,1,0,\ldots,0)$ with number of $1$s equal to $m$; $E_2$ has first $m$ entries $0$, and $C_2$ is of the form $\begin{pmatrix}C_{11}&0\\ C_{21}& C_{22}\end{pmatrix}$ where $C_{11}$ is an $m\times m$ skew-Hermitian.
\end{enumerate}
 
\noindent\textbf{Justification}: Once we have $A_1$ in diagonal form we use ER4 to change $X_1$ and EC4 to change $E_1$. In the first case these can be made $0$, however in the second case we can only make first $m$ entries zero. Then Lemma~\ref{lemmaF} makes sure that $C_1$ has the required form, call it $C_2$.

\item[Step 3:]
\textbf{Input:} Matrix $g_2=\begin{pmatrix} \alpha_2 & 0 & Y_2 \\ E_2& A_2 &B_2\\ F_2& C_2& D_2\end{pmatrix}$.

\noindent\textbf{Output:} 
\begin{enumerate} 
\item[a:] Matrix $g_3=\begin{pmatrix}\alpha_3 & 0 & Y_3 \\ 0& A_3 &B_3\\ F_3& 0 & D_3 \end{pmatrix}$ where $A_3$ is $\diag(1,1,\ldots, 1, \lambda)$. 
\item[b:] Matrix $g_3=\begin{pmatrix}\alpha_3 & 0 & Y_3 \\ E_3 & A_3 &B_3\\ F_3&  0 & D_3 \end{pmatrix}$.
\end{enumerate}
\noindent\textbf{Justification}: Observe the effect of ER3 and EC4. Then Lemma~\ref{lemmaC} ensures that $C_3$ is zero in the first case. In the second case, it only makes first $m$ rows of $C_3$ zero. Thus we interchange remaining rows of $C_3$ with $A_3$ to get the desired result. 
\item[Step 4:] \textbf{Input:} Matrix $g_3=\begin{pmatrix}\alpha_3 & 0 & Y_3 \\ E_3& A_3 &B_3\\ F_3& 0 & D_3 \end{pmatrix}$.

\noindent\textbf{Output:} Matrix $g_4=\begin{pmatrix} \alpha_4 &0&0\\ 0&A_4& B_4 \\0&0&\tra{\bar A}_4^{-1}\end{pmatrix}$  with $A_4=\diag(1,\ldots,1,\lambda)$ and $\alpha_4 \in\mathcal{K}^1$. 

\noindent\textbf{Justification}: In the first case we already have $E_3=0$ thus Lemma~\ref{lemmaG} gives the  desired result. In the second case, if needed we use ER1 and EC1 on $A_3$ to make it a diagonal. Lemma~\ref{lemmaI} ensures that $A_3$ has the full rank. Furthermore, we can use $ER4$ and $EC4$ to make $E_3=0$. Then again Lemma~\ref{lemmaG} gives the required form.
\item[Step 5:] \textbf{Input:} Matrix $g_4=\begin{pmatrix} \alpha_4 &0&0\\ 0&A_4&B_4\\0&0&\bar A_4^{-1}\end{pmatrix}$  with $A_4=\diag(1,\ldots,1,\lambda)$ and $\alpha_4\in \mathcal{K}^1$.

\noindent\textbf{Output:} Matrix $g_5=\diag(\alpha, 1\ldots, 1, \lambda, 1, \ldots, 1,\bar\lambda^{-1})$ where $\alpha\in\mathcal{K}^1$. 

\noindent\textbf{Justification}: Lemma~\ref{lemmaH} ensures that $B_4$ is of a certain kind. We can use ER2 to make $B_4=0$.
\end{itemize}
\subsection{Proof of Theorem A}
\begin{proof}
Let $g\in \text{U}(d,\mathcal{K})$. Using the Gaussian elimination above we
can reduce $g$ to a matrix of the form
$\diag(1,\ldots,1,\lambda,1,\ldots,1,\bar\lambda^{-1})$ when $d=2l$
and $\diag(\alpha, 1\ldots, 1, \lambda, 1, \ldots,
1,\bar\lambda^{-1})$ when $d=2l+1$. We further note that row-column operations are multiplication by elementary matrices from left or right and each of these elementary matrices have determinant one. Thus we get the required result.
\end{proof}

\subsection{Asymptotic complexity is $O(l^3)$} In this section, we show that the
asymptotic complexity of the algorithm that we developed is
$O(l^3)$. We count the number of field multiplications. We can break
the algorithm into three parts. One, reduce $A$ to a diagonal, then
deal with $C$ and then with $D$. It is easy to see that reducing $A$
to the diagonal has complexity $O(l^3)$ and dealing with $C$ and $D$ has complexity
$O(l^3)$. Row interchange has complexity $O(l^2)$.
In the odd-order case there is a complexity of $O(l)$ to
deal with $X,Y,E$ and $F$. In all, the worst case complexity is $O(l^3)$.

\section{A few technical lemmas}
To justify our algorithms we need few lemmas. Many of
these lemmas could be known to an expert. However,  we
include them for the convenience of a reader.

\begin{lemma}\label{lemmaA}
Let $Y=\diag(1,\ldots,1,\lambda,\ldots,\lambda)$ of size $l$ with number of $1$s equal to $m<l$. Let $X$ be a matrix such that $YX$ is skew-Hermitian then $X$ is of the form $\begin{pmatrix}X_{11}& -\bar\lambda \tr {\bar X}_{21}\\ X_{21}&X_{22}\end{pmatrix}$ where $X_{11}$ is skew-Hermitian and so is $\lambda X_{22}$ if $\lambda\neq 0$.
\end{lemma}
\begin{proof}
We observe that the matrix $YX=\begin{pmatrix}X_{11}& X_{12}\\\lambda X_{21}&\lambda X_{22}\end{pmatrix}$. The condition that $YX$ is skew-Hermitian implies $X_{11}$ (and $X_{22}$ if $\lambda\neq 0$) is skew-Hermitian and $X_{12}=-\bar\lambda \tra {\bar X}_{21}$. 
\end{proof}
\begin{corollary}\label{corA}
Let $g=\begin{pmatrix}A&B\\C&D\end{pmatrix}$ be in $\text{U}(2l,\mathcal{K})$.
\begin{enumerate}
\item If $A$ is a diagonal matrix $\diag(1,1,\ldots,1,\lambda)$ with $\lambda\neq 0$ then the matrix $C$ is of the form $\begin{pmatrix}C_{11}& -\lambda \tr {\bar C}_{21}\\ C_{21}& c\end{pmatrix}$ where $C_{11}$ is an $(l-1)\times (l-1)$ skew-Hermitian and $\bar \lambda c\in\mathcal{K}^o$.
\item  If $A$ is a diagonal matrix $\diag(1,1,\ldots,1,0,\ldots,0)$ with number of $1$s equal to $m<l$ then the matrix $C$ is of the form $\begin{pmatrix}C_{11}&0\\ C_{21}& C_{22}\end{pmatrix}$ where $C_{11}$ is an $m\times m$ skew-Hermitian.
\end{enumerate}
\end{corollary}
\begin{proof}
We use the condition that $g$ satisfies $\tra {\bar g}\beta g=\beta$.
\begin{eqnarray*}
\tra {\bar g} \beta g &=& \begin{pmatrix}\tr {\bar A}&\tra {\bar C}\\\tra {\bar B}&\tra {\bar D}\end{pmatrix} \begin{pmatrix}&I\\ I&\end{pmatrix} \begin{pmatrix}A&B\\C&D\end{pmatrix}\\
&=& \begin{pmatrix} \tra {\bar C}&\tr {\bar A}\\ \tra {\bar D}&\tra{\bar B}\end{pmatrix}\begin{pmatrix}A&B\\C&D\end{pmatrix} =\begin{pmatrix} \tra {\bar C}A+\tr {\bar A}C&*\\ *&*\end{pmatrix}
\end{eqnarray*}
This gives $\tra{\bar C}A+\tr{\bar A}C=0$ which means $\bar AC$ is skew-Hermitian (note $A=\tra A$ as $A$ is diagonal). The Lemma~\ref{lemmaA} gives the required form for $C$.
\end{proof}
\begin{corollary}\label{corB}
Let $g=\begin{pmatrix}A&B\\0&\bar A^{-1}\end{pmatrix}$ where $A=\diag(1,\ldots,1,\lambda)$ be an element of $\text{U}(2l,\mathcal{K})$ then the matrix $B$ is of the form  $\begin{pmatrix}B_{11}&-\lambda^{-1} \tra\bar B_{21}\\ B_{21}& b\end{pmatrix}$ where $B_{11}$ is a skew-Hermitian matrix of size $l-1$ and $\bar \lambda^{-1} b\in\mathcal{K}^o$.
\end{corollary}
\begin{proof}
Yet again, we use the condition that $g$ satisfies $\tra {\bar g}\beta g=\beta$ and $A=\tr A$.
\end{proof}
 
\begin{lemma}\label{lemmaB}
 A matrix $\begin{pmatrix} A &B\\0 &D   \end{pmatrix}$ belongs to $\text{U}(2l,\mathcal{K})$ if and only if $D=\tr{\bar A}^{-1}$ and $A^{-1}B$ is skew-Hermitian.
\end{lemma}
\begin{proof}
 The proof is simple computation.
\end{proof}
 
 \begin{lemma}\label{lemmaC}
Let $Y=\diag(1,1,\ldots,1,\bar \lambda)$ be of size $l$ where $\lambda\neq 0$ and $X=(x_{ij})$ be a matrix such that $YX$ is skew-Hermitian. Then $X=(R_1+R_2+\ldots)Y$ where each $R_m$ is of the form $te_{i,j}-\bar te_{j,i}$ with $t\in \mathcal{K}$ for some $i < j$ or of the form $se_{i,i}$ with $s\in\mathcal{K}^o$ for some $i$. 
\end{lemma}
\begin{proof}
Since $YX$ is skew-Hermitian, the matrix $X$ is of the following form (see Lemma~\ref{lemmaA}): 
$\begin{pmatrix}X_{11}&X_{12}\\ X_{21} & x\end{pmatrix}$ where $X_{11}$ is skew-Hermitian of size $(l-1)\times (l-1)$ and  $X_{21}$ is a row of size $l-1$ and $X_{12} = -\lambda \tr {\bar X}_{21}$ and $x=\lambda \frac{x}{\lambda}$ a scalar satisfying $\bar \lambda x\in \mathcal{K}^o$. Clearly any such matrix is sum of the matrices of the form $R_mY$.    
\end{proof}

\begin{lemma}\label{lemmaF}
Let $g=\begin{pmatrix}\alpha&X&*\\ *&A&*\\ *&C&*\end{pmatrix}$ be in $\text{U}(2l+1,\mathcal{K})$.
\begin{enumerate}
\item If $A=\diag(1,\ldots,1,\lambda)$ and $X=0$ then $C$ is of the form $\begin{pmatrix}C_{11}&-\lambda \tra {\bar C}_{21}\\ C_{21}&c\end{pmatrix}$ with $C_{11}$ skew-Hermitian and $\bar \lambda c\in\mathcal{K}^o$.
\item  If $A=\diag(1,\ldots,1,0,\ldots,0)$ with number of $1$s equal $m<l$ and $X$ has first $m$ entries $0$ then $C$ is of the form $\begin{pmatrix}C_{11}& 0 \\ *&*\end{pmatrix}$ with $C_{11}$ skew-Hermitian and $X$ must be zero. 
\end{enumerate}
\end{lemma}
\begin{proof} We use the equation $\tra {\bar g} \beta g=\beta$ and get $2\tra {\bar X}X= -(\bar \tra CA+\tra{\bar A}C)$. In the first case $X=0$, so we can use Corollary~\ref{corA} to get the required form for $C$. In the second case we write $C=\begin{pmatrix}C_{11}&C_{12}\\ C_{21}&C_{22}\end{pmatrix}$ then the equation implies: $-\begin{pmatrix}C_{11}+\tra\bar C_{11}& C_{12}\\ \tra\bar C_{12}& 0\end{pmatrix}=\begin{pmatrix}0&0\\0&2\tra{\bar M}M\end{pmatrix}$ where $M=(x_{m+1},\ldots,x_l)$. This gives the required result. This implies $X=0$ and $C$ has required form.
\end{proof}

\begin{lemma}\label{lemmaI}
Let $g=\begin{pmatrix}\alpha&X&Y\\ *&A&*\\ *&0&D\end{pmatrix}$ be in $\text{U}(2l+1,\mathcal{K})$ then $X=0$ and $D=\tra{\bar A}^{-1}$.
\end{lemma}
\begin{proof} We compute $\tra{\bar g}\beta g=\beta$ and get $2\tra{\bar X}X=0$ and $2\tra{\bar X}Y+\tra{\bar A}D=I$. This gives the required result.
\end{proof}

\begin{lemma}\label{lemmaG}
Let $g=\begin{pmatrix}\alpha& 0&Y\\ 0&A&B\\F&0&D\end{pmatrix}$, with $A$ an invertible diagonal matrix, be in $\text{U}(2l+1,\mathcal{K})$ then $\alpha\bar\alpha=1, F=0=Y$, $D=\bar A^{-1}$ and $\tra{\bar D}B+\tra{\bar B}D=0$. 
\end{lemma}
\begin{proof}
\begin{eqnarray*}
\tra \bar g\beta g &=& \begin{pmatrix}\bar\alpha& 0&\tra{\bar F}\\ 0&\tra\bar A& 0 \\\tra\bar Y& \tra\bar B&\tra\bar D\end{pmatrix} \begin{pmatrix}2& &\\ &&I\\&I&\end{pmatrix}\begin{pmatrix}\alpha& 0&Y\\ 0&A&B\\F&0&D\end{pmatrix}\\
&=& \begin{pmatrix}2\alpha\bar\alpha& \tra\bar F A &2\bar\alpha Y+\tra\bar FB\\ \tra\bar AF &0 &\tra\bar AD\\2\alpha \tra\bar Y +\tra\bar BF& \tra\bar DA & 2\tra\bar YY +\tra\bar DB+\tra\bar BD\end{pmatrix}.
\end{eqnarray*}
Equating this with $\beta$ we get the required result.
\end{proof}

\begin{lemma}\label{lemmaH}
Let $g=\begin{pmatrix}\alpha & 0&0\\ 0&A&B\\0&0&\bar A^{-1}\end{pmatrix}\in \text{U}(2l+1,\mathcal{K})$ where $A=\diag(1,\ldots,1,\lambda)$ is invertible then $B$ is of the form $\begin{pmatrix}B_{11}&-\bar\lambda^{-1}\tra\bar B_{21}\\ B_{21} & b \end{pmatrix}$ where $B_{11}$ is skew-Hermitian and $\bar\lambda^{-1} b\in\mathcal{K}^o$.
\end{lemma}
\section{Finite Unitary Groups}
In the next section, we talk about cryptography. In cryptography, we need to deal
explicitly with finite fields. In this context, when
$\mathcal{K}=\mathbb F_{q^2}$, we prove a theorem similar in spirit to
Steinberg~\cite[Section 6.2]{st3}. The proof is an obvious corollary
to our algorithm.
\begin{theorem}\label{main-cor}
Fix an element $\zeta$ which generates the cyclic group $\mathbb
F_{q^2}^{\times}$, the subgroup $\mathbb F_{q^2}^1$ is generated by
$\zeta_1=\zeta^{q-1}$. We add following matrices to the respective set of elementary matrices:
\begin{itemize}
\item $h\left(\zeta\right)=\diag(1,\ldots,\zeta,1,\ldots,\bar{\zeta}^{-1})$
  whenever $d=2l$.
\item $\left\{\begin{array}{cc}
h\left(\zeta\right)=&\diag(1,\ldots,\zeta,1,\ldots,\bar{\zeta}^{-1})\\
h\left(\zeta_1\right)=&\diag(\zeta_1,1,\ldots,1)
\end{array}
\right.$ whenever $d=2l+1$.
\end{itemize}
Then the group U$(d,q^2)$ is generated by elementary matrices and the
matrices defined above. 
\end{theorem}

\subsection{Special Unitary group $\text{SU}(d,q^2)$}
In the case of $\text{SU}(2l,q^2)$ a simple and straightforward
enhancement of our algorithm reduces a matrix
$g\in\text{SU}(2l,q^2)$ to the identity matrix. Thus the word problem
in $\text{SU}(2l,q^2)$ is completely solved as with $\text{SL}(d,q)$
using only elementary matrices;
this is particularly useful for the MOR cryptosystem. An analysis of a MOR cryptosystem
similar to the MOR cryptosystem over $\text{SL}(d,q)$~\cite{Ma} will
be done in the next section. 

For the reduction to identity, note that Theorem~\ref{theorema} reduces $g$ to $\diag(1, \ldots, 1, \lambda, 1\ldots, 1, \bar\lambda^{-1})$. However, since $\det(g)=\lambda\bar\lambda^{-1}=1$, we have $\lambda=\bar\lambda$ and $\lambda\in \mathbb F_q^{\times}$. Now, for $s=\varepsilon \lambda$ where $\bar\varepsilon=-\varepsilon$,
$$w_{l,-l}(s)w_{l,-l}(-\varepsilon)=\diag(1,\ldots,1,\lambda,1,\ldots,\lambda^{-1}).$$
So if we add $w_{l,-l}(s)w_{l,-l}(-\varepsilon)$ to the output of
Theorem~\ref{theorema}, we have the identity matrix.

In the case of $\text{SU}(2l+1,q^2)$ we need to add an extra generator
$h(\zeta_1)=\diag(\zeta_1,1,\ldots,1)$ where $\zeta_1$ is a generator of
$\mathbb F_{q^2}^1$. Now we can reduce an element of the form $
\diag(\alpha,1\ldots,\lambda,1,\ldots,\bar\lambda^{-1})$ to $
\diag(1,1\ldots,\lambda,1,\ldots,\bar\lambda^{-1})$ by multiplying
with the suitable power of $h(\zeta_1)$. Note that finding the suitable power involves solving a discrete logarithm problem.
Then we use similar computations for
even-order case to reduce  $\diag(1,1\ldots,\lambda,1,\ldots,\bar\lambda^{-1})$ to identity.  

\section{The MOR cryptosystem on unitary groups}
Briefly speaking, the MOR cryptosystem is a simple and straightforward
generalization of the classic ElGamal cryptosystem and was put forward by
Paeng et.~al.~\cite{crypto2001}. In a MOR cryptosystem one works with
the automorphism group rather than the group itself. It provides an
interesting change in perspective in public-key cryptography -- from
finite cyclic groups to finite non-abelian groups. The MOR cryptosystem was studied for the special linear group in details by Mahalanobis~\cite{Ma}. For many other classical groups, except the orthogonal groups, the analysis of a MOR cryptosystem remains almost the same. So we will remain brief in this paper and refer an interested reader to ~\cite{Ma}.

\noindent The description of the MOR cryptosystem is as follows:

Let $G=\langle g_1,g_2,\ldots,g_s\rangle$ be a finite group. Let $\phi$ be a non-identity automorphism.
\begin{itemize}
\item \textbf{Public-key:} Let $\{\phi(g_i)\}_{i=1}^s$ and $\{\phi^{\mathfrak{m}}(g_i)\}_{i=1}^s$ is public.
\item \textbf{Private-key:} The integer $\mathfrak{m}$ is private.
\end{itemize}
\textbf{Encryption:}\leavevmode\newline
To encrypt a plaintext $\mathfrak{M}\in G$, get an arbitrary integer $r\in[1,|\phi|]$ compute $\phi^r$ and $\phi^{r\mathfrak{m}}$. 
The ciphertext is $\left(\phi^r,\phi^{r\mathfrak{m}}\left(\mathfrak{M}\right)\right)$.\newline
\textbf{Decryption:}\leavevmode\newline
After receiving the ciphertext $\left(\phi^r,\phi^{r\mathfrak{m}}\left(\mathfrak{M}\right)\right)$, the user knows the private key $\mathfrak{m}$. So she computes $\phi^{mr}$ from $\phi^r$ and then computes $\mathfrak{M}$. 

To develop a MOR cryptosystem we need a thorough understanding of the
automorphisms group of the group involved. The automorphisms of unitary groups are well
described in the literature. We mention them briefly to facilitate
further discussion. 
\subsection{Automorphism Group of Unitary Groups}\label{automorphismclassical}
First we define the similitude
group. We need these groups to define diagonal automorphisms.
\begin{definition}[Unitary similitude group]
 The unitary similitude group is defined as:
$$\text{GU}(d,q^2)=\{X\in \text{GL}(d,q^2) \mid \tr{\bar X}\beta X=\mu\beta, \text{for\ some\ } \mu\in \mathbb F_q^{\times}\}.$$
\end{definition}
\noindent Note that the multiplier $\mu$ defines a group homomorphism from $\text{GU}(d,q^2)$ to $\mathbb F_q^{\times}$ with kernel the unitary group.

{\bf Conjugation Automorphisms: } The conjugation maps $g\mapsto ngn^{-1}$ for $n\in\text{GU}(d,q^2)$ are called conjugation automorphisms. Furthermore,
they are composition of two types of automorphisms -- inner automorphisms given as conjugation by elements of $\text{U}(d,q^2)$ and diagonal automorphisms given as conjugation by diagonals of  $\text{GU}(d,q^2)$.

{\bf Central Automorphisms: } Let $\chi \colon \text{U}(d,q^2)
\rightarrow \mathbb F_{q^2}^1$ be a group homomorphism. Then the
central automorphism $c_{\chi}$ is given by $g\mapsto \chi(g)g$. Since
$[\text{U}(d,q^2), \text{U}(d,q^2)]= [\text{SU}(d,q^2),
\text{SU}(d,q^2)]=  \text{SU}(d,q^2)$~\cite[Theorem 11.22]{gr}, any
$\chi$ is equivalent to a group homomorphism from
$\text{U}(d,q^2)/\text{SU}(d,q^2)$ to $\mathbb F_{q^2}^1$. There are
at most $q+1$ such maps. 

{\bf Field Automorphisms: } For any automorphism $\sigma$ of the field
$\mathbb F_{q^2}$, replacing all entries of a matrix by their image
under $\sigma$ give us a field automorphism.

The following theorem, due to Dieudonn\'{e}~\cite[Theorem 25]{di}, describes all automorphisms:
\begin{theorem}
 Let $q$ be odd and $d\geq 4$. Then any automorphism $\phi$ of the unitary group $\text{U}(d,q^2)$ is written as $c_{\chi}\iota\delta\sigma$ where $c_{\chi}$ is a central automorphism, $\iota$ is an inner automorphism, $\delta$ is a diagonal automorphism and
 $\sigma$ is a field automorphism.
\end{theorem}

As we saw above there are three kind of automorphisms in an unitary group. One is conjugation automorphism, the others are central and field automorphisms. A central automorphism being multiplication by an element of the center, that is a field element. Exponentiation of that will be a discrete logarithm problem in $\mathbb{F}_q$. Similar is the case with a field automorphism. So the only choice for a better MOR cryptosystem is a conjugating automorphism.

Once, we have decided that the automorphism that we are going to use in the MOR cryptosystem will act by conjugation. Further analysis is straightforward and follows~\cite[Section 7]{Ma}. Recall that we insisted that automorphisms in the MOR cryptosystem are presented as action on generators. In this case, the generators are \emph{elementary matrices} and the group is a \emph{special unitary group} of even-order. Other groups can be used and analyzed similarly. Note that two things can happen: one can find the conjugator element for the automorphism in use, finding the conjugator up to a scalar multiple is enough. Or one cannot find the conjugator in the automorphism.

In the first case, the discrete logarithm problem in the automorphism becomes a discrete logarithm problem in a matrix group. Assume that we found the conjugating matrix $A$ up to a scalar multiple, where $A\in\text{GU}(d,q^2)$. Now the discrete logarithm problem in $\phi$ becomes a discrete logarithm problem in $A$. One can show that by suitably choosing $A$, the discrete logarithm in $A$ is embedded in the field $\mathbb{F}_{q^{2d}}$. This argument is presented in details~\cite[Section 7.1]{Ma}. We will not repeat it here. In the next section (reduction of security), we show that one can find this conjugating element for unitary groups and this gives us a proof of Theorem B. 
 
The success of any cryptosystem comes from a balance between speed and
security. In this paper, we deal with both speed and security of the MOR
cryptosystem briefly. For an implementation of the MOR cryptosystem, we need to
compute power of an automorphism. The algorithm of our choice is the
famous \textbf{square-and-multiply} algorithm. Since we do not use any
special algorithm for squaring, squaring and multiplying is the same
for us. So we talk about multiplying two automorphisms. We present the automorphisms as action on generators, i.e., $\phi(g_i)$ is a matrix for $i=1,2,\ldots,s$. The first step of the algorithm is to find the word in generators from the matrix\footnote{One can also present the automorphisms as word in generators, we choose matrices.}. So now the automorphism is $\phi(g_i)=w_i$ where each $w_i$ is a word in generators.
Once that is done then composing with an automorphism is substituting
each generator in the word by another word. This can be done fast. The
challenging thing is to find the matrix corresponding to the word thus
formed. This is not a hard problem, but can be both time and memory
intensive. What is the best way to do it is still an open question!
However, there are many shortcuts available. One being an obvious
time-memory trade off, like storing matrices corresponding to a word
in generators. The other being there are many trivial and non-trivial
relations among these generators and moreover these generators are
sparse matrices. One can use these properties in the implementation. 

This problem, which is of independent interest in computational group
theory and is the reason that we insist on automorphisms being
presented as generators for the MOR cryptosystem. For more information, see~\cite[Section 8]{Ma}.   

\subsection{Reduction of security}

In this subsection, we show that for unitary groups, the security of the MOR cryptosystem is the hardness of the discrete logarithm problem in $\mathbb{F}_{q^{2d}}$. This is the same as saying that we can find the conjugating matrix up to a scalar multiple. Let $\phi$ be an automorphism that works by conjugation, i.e., $\phi=\iota_g$ for some $g$ and we try to determine $g$.

{\bf Step 1:}  The automorphism $\phi$ is presented as action on generators. Thus $\phi(x_{i,-i}(s))=g(I+se_{i,-i})g^{-1}=I+sge_{i,-i}g^{-1}$. This implies that we know $\varepsilon ge_{i,-i}g^{-1}$ and similarly $\varepsilon ge_{-i,i}g^{-1}$ for fixed $s=\varepsilon$. We first claim that we can determine $N:=gD$ where $D$ is diagonal.

When $d=2l$, write $g$ in the column form $\left[G_1,\ldots G_l, G_{-1},\ldots, G_{-l}\right]$. 
Now, 
\begin{enumerate}
\item $\left[G_1,\ldots G_l, G_{-1},\ldots, G_{-l}\right]\varepsilon e_{i,-i}=\left[0,\ldots,0,\varepsilon G_{i},0,\ldots,0\right] $
where $G_{i}$ is at $-i\textsuperscript{th}$ place. Multiplying this with $g^{-1}$ gives us scalar multiple of $G_i$, say $d_i$.
\item $\left[G_1,\ldots G_l, G_{-1},\ldots, G_{-l}\right]\varepsilon e_{-i,i}=\left[0,\ldots,0,\varepsilon G_{-i},0,\ldots,0\right] $
where $G_{-i}$ is at $i\textsuperscript{th}$ place. Multiplying this with $g^{-1}$ gives us scalar multiple of $G_{-i}$, say $d_{-i}$.
\end{enumerate}
Thus we get $N=gD$ where $D$ is a diagonal matrix $\diag(d_1,\ldots, d_l,d_{-1},\ldots,d_{-l})$.
In the case when $d=2l+1$ we write
$g=[G_0,G_1,\ldots,G_l,G_{-1},\ldots,G_{-l}]$ and get scalar multiple
of columns $G_i$ and $G_{-i}$. We now use $x_{i,0}(t)$ and $x_{0,i}(t)$ to get linear combination of $G_0$ with $G_{i}$ or $G_{-i}$, say we get $\alpha G_0+\beta G_{-1}$.  In this case we get $N=gD$ where $D$ is of the form $$\begin{pmatrix} \alpha &&&&&&\\ &d_1&&&&&\\ &&\ddots &&&& \\ &&&d_l&&&\\ \beta &&&&d_{-1}&&\\ &&&&&\ddots&\\ &&&&&&d_{-l}\end{pmatrix}.$$

{\bf Step 2:}
Now we compute $N^{-1}\phi(x_r(t))N=D^{-1}g^{-1}(gx_r(t)g^{-1})gD=D^{-1}x_r(t)D$. Substituting various $x_r(t)$ it amounts to computing $D^{-1}e_rD$. 
When $d=2l$, we first compute $D^{-1}(e_{i,j}-e_{-j,-i})D$ and get $d_i^{-1}d_j$, $d_{-i}^{-1}d_{-j}$ for $i\neq j$. 
Then we compute $D^{-1}e_{i,-i}D, D^{-1}e_{-i,i}D$ and get $d_id_{-i}^{-1}, d_{-i}d_{i}^{-1}$. We form a matrix 
$$\diag(1,d_2^{-1}d_1,\ldots,d_l^{-1}d_1, d_{-1}^{-1}d_1,\ldots,d_{-l}^{-1}d_1)$$
and multiply it to $N=gD$ to get $d_1g$. Thus we can determine $g$ up to a scalar multiple and the attack follows~\cite[Section 7.1.1]{Ma}. 

In the case $d=2l+1$, the matrix $D$ is almost a diagonal matrix except the first column. However while computing $D^{-1}(e_{12}-e_{-2,-1})D$ we also get $d_{-2}^{-1}\beta$ and by computing $D^{-1}(e_{0,1}-2e_{-1,0}-e_{-1,1})D$ we get $\alpha^{-1}D$. Thus we can multiply $\alpha G_0+\beta G_{-1}$ 
by $\beta^{-1}d_1= \beta^{-1}d_{-2}d_{-2}^{-1}d_{-1}d_{-1}^{-1}d_1$ and get $\alpha \beta^{-1}d_1 G_0 + d_1G_{-1}$. With the computation in even case we can determine $d_{1}G_{-1}$ and hence can determine $\alpha G_0$. Furthermore, since we know $\alpha^{-1}d_1$ we can determine $d_1G_0$ thus in this case as well we can determine $d_1g$, i.e., $g$ up to a scalar multiple.


\section{Conclusion}
For us, this paper is an interplay of finite (non-abelian) groups and
public key cryptography. Computational group theory, in particular
computations with quasi-simple groups have a long and distinguished
history~\cite{praeger,kantor,CMT,lo,all}. The interesting thing to us is, some of
the questions that arise naturally when dealing with the MOR
cryptosystem are interesting in its own right in computational group
theory and are actively studied. The row-column operations that we
developed is one example of that. In the row-column operations we
developed, we used a different set of generators. These generators
have a long history starting with Chevalley. In our knowledge, we are
the first to use them in row-column operations in Unitary
groups. Earlier works were mostly done using the standard
generators. It seems that Chevalley generators might offer a paradigm
shift in algorithms with quasi-simple groups. In Magma, there is an
implementation of row-column operations in unitary groups in a function
\emph{ClassicalRewriteNatural}. We compared
that function with our algorithm in an actual implementation on even
order unitary groups. To select parameters for our simulation, we followed Costi's work~\cite[Table 6.2]{costi}. In one case, the characteristic of the field was fixed at $7$ and the size of the matrix at $20$, we varied the degree of the extension of the field from $4$ to $34$. We then picked at random elements from the GeneralUnitaryGroup and timed our algorithm. The final time was the average over one thousand repetitions. We did the same with the magma function using special unitary groups. Times of both these algorithms were tabulated and is presented in Figure 1. In another case, we kept the field fixed at $7^{10}$ and changed the size of the matrix. In all cases, the final time was the average of one thousand random repetitions. The timing was tabulated and presented in Figure 1. It seems the our algorithm is much better than that of Costi's from all aspects.
\begin{figure}
\centering
\includegraphics[scale=0.5]{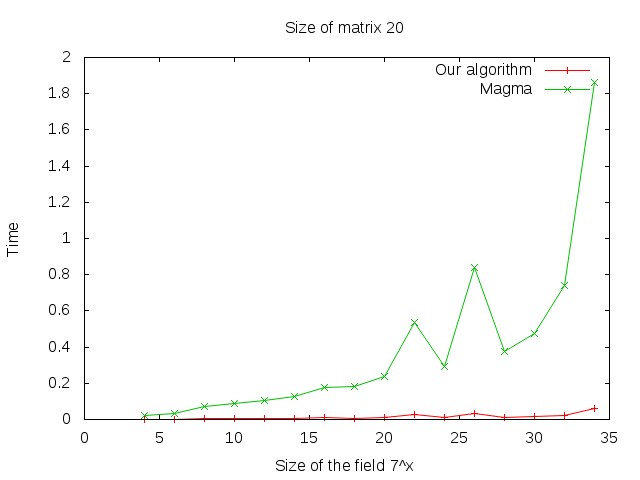}
\includegraphics[scale=0.5]{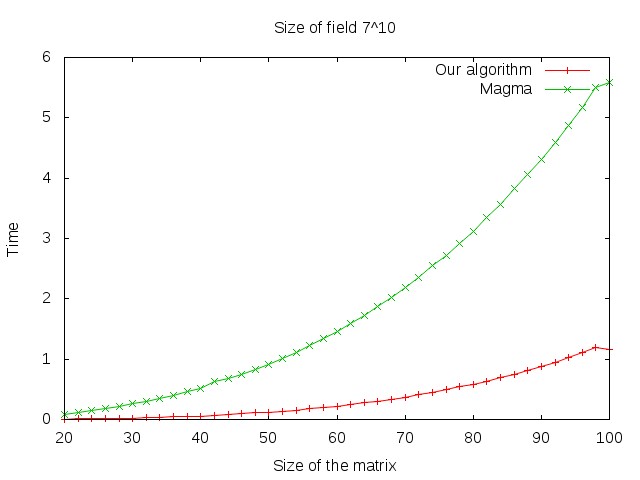}
\caption{Some simulations comparing our algorithm with the one inbuilt in Magma}
\end{figure}
\nocite{brooksbank} 
\bibliographystyle{amsplain}
\bibliography{paper}
\end{document}